\documentclass[11pt]{amsart}
\usepackage{epsfig}

\theoremstyle{plain}
\newtheorem{theorem}{Theorem}[section]

\newtheorem{lemma}[theorem]{Lemma}
\newtheorem{corollary}[theorem]{Corollary}
\newtheorem{definition}[theorem]{Definition}

\theoremstyle{remark}
\newtheorem{remark}[theorem]{Remark}

\begin{document}

\title[Unstabilized critical Heegaard surfaces]
{Examples of unstabilized critical Heegaard surfaces}

\author{Jung Hoon Lee}
\address{School of Mathematics, KIAS\\
207-43, Cheongnyangni 2-dong, Dongdaemun-gu\\
Seoul, Korea}
\email{jhlee@kias.re.kr}

%\thanks{}


\begin{abstract}
We show that the standard minimal genus Heegaard splitting of (closed orientable surface)$\times S^1$ is a critical Heegaard splitting.
\end{abstract}

\maketitle

\section{Introduction}

Let $S$ be a closed orientable connected separating surface in an irreducible $3$-manifold $M$, dividing $M$ into two manifolds $V$ and $W$. Define the {\it disk complex} $\mathcal{D}(S)$ as follows.
\begin{enumerate}
\item Vertices of $\mathcal{D}(S)$ are isotopy classes of essential simple closed curves in $S$ that bound disks in $V$ or $W$.
\item $k+1$ distinct vertices constitute a $k$-cell if there are disjoint representatives.
\end{enumerate}

By abuse of terminology, we sometimes identify a vertex with some representative essential disk of the vertex.
Let $\mathcal{D}_V(S)$ and $\mathcal{D}_W(S)$ be the subcomplex of $\mathcal{D}(S)$ spanned by vertices that bound disks in $V$ and $W$ respectively.
In this paper, we only consider vertices and edges of $\mathcal{D}_V(S)$, $\mathcal{D}_W(S)$ and $\mathcal{D}(S)$.

$S$ is {\it strongly irreducible} if $\mathcal{D}_V(S)$ is not connected to $\mathcal{D}_W(S)$ in $\mathcal{D}(S)$. Strongly irreducible surfaces were proved to be useful to analyze the Heegaard structure and topology of $3$-manifolds. For example, if a minimal genus Heegaard surface is not strongly irreducible, then the manifold contains an incompressible surface \cite{Casson-Gordon}.

Incompressible surface and strongly irreducible surface can be regarded as a topological analogue of an index $0$ minimal surface and index $1$ minimal surface  respectively \cite{Bachman3}. As a generalization of this idea, Bachman has defined a notion of {\it critical surface} \cite{Bachman1}, \cite{Bachman2}, which can be regarded as a topological index $2$ minimal surface.

\begin{definition}
$S$ is critical if vertices of $\mathcal{D}(S)$ can be partitioned into two sets $C_0$ and $C_1$.

\begin{enumerate}
\item For each $i=0,1$, there is at least one pair of disks $V_i\in \mathcal{D}_V(S)\cap C_i$ and $W_i\in \mathcal{D}_W(S)\cap C_i$ such that $V_i\cap W_i=\emptyset$.
\item If $V_i\in \mathcal{D}_V(S)\cap C_i$ and $W_{1-i}\in \mathcal{D}_W(S)\cap C_{1-i}$, then $V_i\cap W_{1-i}\ne \emptyset$ for any representatives. In other words, $V_i$ and $W_{1-i}$ are not joined by an edge.
\end{enumerate}
\end{definition}

Critical surfaces have some useful properties. For example, if an irreducible manifold contains an incompressible surface and a critical surface, then the two surfaces can be isotoped so that any intersection loop is essential on both surfaces.

In \cite{Bachman1}, it is shown that if a manifold which does not contain incompressible surfaces has two distinct Heegaard splittings, then the minimal genus common stabilization of the two splittings is critical.
In this paper, we give concrete examples of critical Heegaard surfaces which are of minimal genus Heegaard splittings, hence unstabilized.

\begin{theorem}
The standard minimal genus Heegaard splitting of (closed orientable surface)$\times S^1$ is a critical Heegaard splitting.
\end{theorem}

\begin{remark}
In \cite{Bachman-Johnson}, for any $n$, Bachman and Johnson constructed examples of manifolds $M_n$ admitting an index $n$ topologically minimal genus $n+1$ Heegaard surface by $n$-fold covering of certain $2$-bridge link exteriors. When $n=2$, the manifold $M_2$ seems to be a genus three manifold by appealing to the Kobayashi's characterization of genus two manifolds containing incompressible tori \cite{Kobayashi}. So existence of minimal genus critical Heegaard surfaces of genus three are already known. Theorem 1.2 gives some examples of higher genus minimal genus critical Heegaard surfaces.
\end{remark}

\begin{corollary}
For any odd $g>1$, there exists a minimal genus critical Heegaard suface of genus $g$.
\end{corollary}

In Section $2$, we give a partition of a disk complex for the standard genus three Heegaard surface of (torus)$\times S^1$. In Section $3$, we show that the partition satisfies the definition of a critical surface.
The same arguments apply for (closed orientable surface)$\times S^1$.

\section{Partition of disk complex}

Let $C$ be a cube $\{(x,y,z)\in \mathbb{R}^3\,|\,-1\le x,y,z\le 1\,\}$. If we identify the three pairs of opposite faces of $C$, we get a $3$-torus $M=$ (torus)$\times S^1$. Let $q:C\rightarrow M$ be the quotient map. An image by $q$ of a tubular neighborhood of union of three axis in $C$ is a genus three handlebody $V$. It is easy to see that $W=cl(M-V)$ is also a genus three handlebody. This decomposition $M=V\cup_S W$ is a standard genus three Heegaard splitting of (torus)$\times S^1$.

\begin{figure}[h]
   \centerline{\includegraphics[width=11cm]{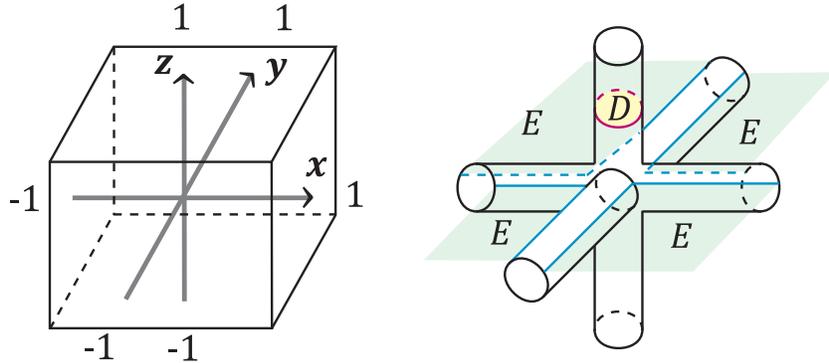}}
    \caption{$D,E\in C_0$ and $D\cap E=\emptyset$}
\end{figure}

Let $D$ be a meridian disk $q(\{z=\frac{1}{2}\}\cap C)\cap V$ in $V$. Let $E$ be a meridian disk $q(\{z=0\}
\cap C)\cap W$ in $W$. See Figure 1. Now we define a partition of vertices of the disk complex $\mathcal{D}(S)=C_0\,\dot{\cup}\, C_1$ as follows.

\begin{enumerate}
\item Let $\mathcal{D}_V(S)\cap C_0$ be essential disks in $V$ that are disjoint from $E$.
\item Let $\mathcal{D}_W(S)\cap C_0$ be essential disks in $W$ that are disjoint from $D$.
\item Let $\mathcal{D}_V(S)\cap C_1$ be $\mathcal{D}_V(S)-(\mathcal{D}_V(S)\cap C_0)$.
\item Let $\mathcal{D}_W(S)\cap C_1$ be $\mathcal{D}_W(S)-(\mathcal{D}_W(S)\cap C_0)$.
\end{enumerate}

Since $V\cup_S W$ is not a reducible Heegaard splitting, the four sets $\mathcal{D}_V(S)\cap C_0$, $\mathcal{D}_W(S)\cap C_0$, $\mathcal{D}_V(S)\cap C_1$, $\mathcal{D}_W(S)\cap C_1$ are mutually disjoint.

Note that $D\in \mathcal{D}_V(S)\cap C_0$ and $E\in \mathcal{D}_W(S)\cap C_0$ and $D\cap E=\emptyset$.
A meridian disk $D'=q(\{x=\frac{1}{2}\}\cap C)\cap V$ is in $\mathcal{D}_V(S)\cap C_1$. A meridian disk $E'=q(\{x=0\}\cap C)\cap W$ is in $\mathcal{D}_W(S)\cap C_1$ and $D'\cap E'=\emptyset$. See Figure 2. So the partition satisfies the first condition of definition of a critical surface.

Since $\mathcal{D}_V(S)\cap C_0$ and $\mathcal{D}_W(S)\cap C_0$ are symmetric, we only need to show that vertices in $\mathcal{D}_V(S)\cap C_0$ and $\mathcal{D}_W(S)\cap C_1$ cannot be joined by an edge.
By definition, any representative of a disk in $\mathcal{D}_W(S)\cap C_1$ intersects $D$. In next section, we will show that $\mathcal{D}_V(S)\cap C_0$ consists of a single vertex $D$, then it completes the proof for the case of (torus)$\times S^1$.

\begin{figure}[h]
   \centerline{\includegraphics[width=5cm]{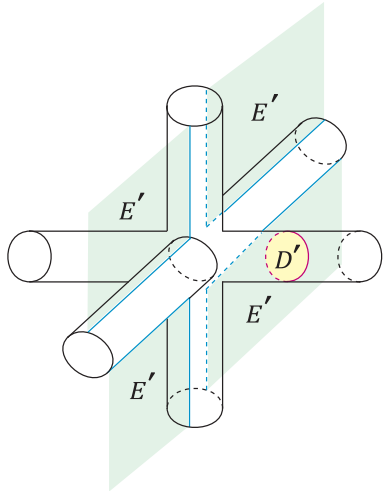}}
    \caption{$D',E'\in C_1$ and $D'\cap E'=\emptyset$}
\end{figure}

\section{Any essential disk in $V$ disjoint from $E$ is $D$}

\begin{lemma}
Any essential disk in $V$ disjoint from $E$ is isotopic to $D$.
\end{lemma}

\begin{proof}
Let $F$ be an essential disk in $V$ disjoint from $\partial E$ which intersects $D$ minimally. Since $V$ is irreducible, we may assume that $F\cap D$ has no circle component of intersection, so the intersection consists of arcs. Let $\gamma$ be an outermost arc of $F\cap D$ in $F$ and $\Delta$ be the corresponding outermost disk in $F$. Let $\Delta'$ be one of the two disks cut by $\gamma$ in $D$.

By slightly pushing $\Delta'$, $\Delta\cup\Delta'$ is disjoint from $D$ and $\partial E$. If we cut $V$ by $D$, we get a genus two handlebody $V'$. $\partial E$ separates $\partial V'$ into two once-punctured tori $T_1$ and $T_2$. We can see that $V'$ is homeomorphic to $T_1\times I$. Since there can be no essential disk in $T_1\times I$ whose boundary is disjoint from $\partial T_1$, $\Delta\cup\Delta'$ is an inessential disk in $V'$. The disk that $\partial (\Delta\cup\Delta')$ bounds in $\partial V'$ may or may not contain a copy of $D$, but in any case we can reduce the intersection of $F\cap D$ by isotopy of $\Delta$.

So we conclude that $F\cap D=\emptyset$. By the above arguments and since $F$ is essential in $V$, $F$ is isotopic to $D$.
\end{proof}

A closed orientable surface $S$ of genus $g>1$ can be obtained from a $4g$-gon $P=a_1 b_1 a^{-1}_1 b^{-1}_1\cdots a_g b_g a^{-1}_g b^{-1}_g$ by identifying corresponding sides. Then $M=S\times S^1$ can be obtained from $C=P\times [-1,1]$ by identifying corresponding pairs of faces. Let $q:C\rightarrow M$ be the quotient map. A standard minimal genus Heegaard splitting $M=V\cup W$ can be obtained similarly as in the case of (torus)$\times S^1$ in Section $2$. Take the center point $p$ of $P\times\{0\}$. Connect $p$ to each side of $P\times\{0\}$, and also connect $p$ to the center of $P\times\{-1\}$ and $P\times\{1\}$ by vertical arcs. Let $V$ be an image by $q$ of a tubular neighborhood of the resulting graph. Let $W=cl(M-V)$.

Let $D$ be a meridian disk $q(P\times \{\frac{1}{2}\})\cap V$ in $V$. Let $E$ be a meridian disk $q(P\times \{0\})\cap W$ in $W$.

$V$ can be regarded as a ((once-punctured surface) $\times I$) $\cup$ ($1$-handle), where $E$ corresponds to the puncture and $D$ corresponds to the co-core of $1$-handle. By similar arguments as in the proof of Lemma 3.1, any essential disk in $V$ disjoint from $E$ is isotopic to $D$. Note that $W$ can also be regarded as a ((once-punctured surface) $\times I$) $\cup$ ($1$-handle), where $D$ corresponds to the puncture and $E$ corresponds to the co-core of $1$-handle. So any essential disk in $W$ disjoint from $D$ is isotopic to $E$.

Hence if we partition the disk complex of $\partial V=\partial W$ as in Section 2, then we can see that the Heegaard surface of (closed orientable surface)$\times S^1$ is a critical Heegaard surface. This completes the proof of Theorem 1.2.

\end{document}